\documentclass[11pt,reqno]{amsart}
\usepackage{amsthm,amssymb,amsmath}
\usepackage{color}

\newtheorem{theorem}{Theorem}
\newtheorem*{theorem*}{Theorem}
\newtheorem{lemma}{Lemma}
\newtheorem{corollary}{Corollary}

\theoremstyle{definition}

\newtheorem*{definition*}{\sc Definition}

\newtheorem*{remark*}{\sc Remark}
\newtheorem*{remarks}{\sc Remarks}

\newtheorem*{example*}{\sc Example}

\newcommand{\thistheoremname}{}

\newtheorem*{genericthm*}{\thistheoremname}
\newenvironment{namedthm*}[1]
  {\renewcommand{\thistheoremname}{#1}%
   \begin{genericthm*}}
  {\end{genericthm*}}

\newcommand{\loc}{{\rm loc}}

\newcommand{\clos}{{\rm clos}}

\begin{document}

\title{$W^{1,p}$ regularity of solutions to Kolmogorov equation  
and associated Feller semigroup}

\author{D.\,Kinzebulatov and Yu.\,A.\,Sem\"{e}nov}

\address{Universit\'{e} Laval, D\'{e}partement de math\'{e}matiques et de statistique, 1045 av.\,de la M\'{e}decine, Qu\'{e}bec, QC, G1V 0A6, Canada {\sf http://archimede.mat.ulaval.ca/pages/kinzebulatov}}

\email{damir.kinzebulatov@mat.ulaval.ca}

\thanks{The research of the first author is supported in part by NSERC}

\address{University of Toronto, Department of Mathematics, 40 St.\,George Str, Toronto, ON, M5S 2E4, Canada}

\email{semenov.yu.a@gmail.com}



\begin{abstract}
 In $\mathbb R^d$, $d \geq 3$, consider the divergence and the non-divergence form operators 
\begin{equation}
\label{div0_}
\tag{$i$}
- \nabla \cdot a \cdot \nabla + b \cdot \nabla, 
\end{equation}
\begin{equation}
\label{nondiv0_}
\tag{$ii$}
- a \cdot \nabla^2 + b \cdot \nabla,
\end{equation}
where $a=I+c \mathsf{f} \otimes \mathsf{f}$, the vector fields $\nabla_i \mathsf{f}$ ($i=1,2,\dots,d$) and $b$ are form-bounded (this includes the sub-critical class $[L^d + L^\infty]^d$ as well as vector fields having critical-order singularities).
We characterize quantitative dependence on $c$ and the values of the form-bounds of the $L^q \rightarrow W^{1,qd/(d-2)}$ regularity of the resolvents of the operator realizations of \eqref{div0_}, \eqref{nondiv0_}  in $L^q$, $q \geq 2 \vee ( d-2)$ as (minus) generators of positivity preserving $L^\infty$ contraction $C_0$ semigroups. The latter allows to run an iteration procedure $L^p \rightarrow L^\infty$ that yields associated with \eqref{div0_}, \eqref{nondiv0_} $L^q$-strong Feller semigroups. 
\end{abstract}

\maketitle

\textbf{1.~}Consider in $\mathbb R^d$, $d \geq 3$, the formal differential operator
\begin{equation}
\label{div0}
-\nabla \cdot a \cdot \nabla + b \cdot \nabla \equiv -\sum_{i,j=1}^d \nabla_i\, a_{ij}(x) \nabla_j + \sum_{j=1}^d b_j(x) \nabla_j,
\end{equation}
where
\begin{equation}
\label{H}
\tag{$H_u$}
\begin{array}{l}
a=a^*:\mathbb R^d \rightarrow \mathbb{R}^d \otimes \mathbb{R}^d \quad \text{is $\mathcal L^d$ measurable}, \\
\sigma I \leq a(x)  \leq \xi I \quad \text{ for $\mathcal L^d$ a.e.\,$x \in \mathbb R^d$ for some $0<\sigma \leq \xi<\infty$}.
\end{array}
\end{equation}
By the De Giorgi-Nash theory, the solution $u \in W^{1,2}(\mathbb R^d)$ to the corresponding elliptic equation $(\mu -\nabla \cdot a \cdot \nabla + b \cdot \nabla) u = f$, $\mu>0$, $f \in L^p\cap L^2$, $p \in ]\frac{d}{2},\infty[$, is in $C^{0,\gamma}$, where the H\"{o}lder continuity exponent $\gamma \in ]0,1[$ depends only on $d$ and $\sigma$, $\xi$, provided that $b:\mathbb R^d \rightarrow \mathbb R^d$ is in the Nash class ($\supset [L^p + L^\infty]^d$, $p>d$) \cite{Se}, but already the class $[L^d + L^\infty]^d$ is not admissible (e.g.\,it is easy to find $b \in [L^d + L^\infty]^d$ that makes the two-sided Gaussian bounds on the fundamental solution of \eqref{div0} invalid).
On the other hand, for $-\Delta +b \cdot \nabla$, the $C^{0,\gamma}$ regularity of solutions to the corresponding elliptic equations 
is known to hold for $b$ having much stronger singularities. Recall that a vector field $b:\mathbb R^d \rightarrow \mathbb R^d$ is in the class  of form-bounded vector fields $\mathbf{F}_\delta \equiv \mathbf{F}_\delta(-\Delta)$, $\delta>0$ if $|b| \in L^2_{\loc} \equiv L^2_{\loc}(\mathbb R^d)$ and 
there exist a constant $\lambda=\lambda_\delta>0$ such that 
$$
\||b|(\lambda -\Delta)^{-\frac{1}{2}}\|_{2\rightarrow 2}\leqslant \sqrt{\delta}.
$$
(The class $\mathbf{F}_\delta$ contains $[L^d + L^\infty]^d$ with $\delta$ arbitrarily small, as follows by  the Sobolev Embedding Theorem, as well as vector fields having critical-order singularities such as $b(x)=\frac{d-2}{2}\sqrt{\delta}|x|^{-2}x$ (by Hardy's inequality) or, more generally, vector fields in $[L^{d,\infty}+L^\infty]^d$, the Campanato-Morrey class or the Chang-Wilson-T.\,Wolff class, with $\delta$ depending on the norm of the vector field in these classes, see e.g.\,\cite{KiS} for details.)
It has been established in \cite{KS} that if $b \in \mathbf{F}_\delta$, $\delta<1$,
then for every $q \in [2,2/\sqrt{\delta}[$ $-\Delta +b \cdot \nabla$ has an operator realization $\Lambda_q(b)$ on $L^q$ as the generator  of a positivity preserving, $L^\infty$ contraction, quasi contraction $C_0$ semigroup $e^{-t\Lambda_q(b)}$  such that $u:=(\mu+\Lambda_q(b))^{-1}f$, $f \in L^q$ satisfies
\begin{equation*}
\|\nabla u\|_q \leq K_1(\mu-\mu_0)^{-\frac{1}{2}}\|f\|_q,  \qquad
\|\nabla|\nabla u|^\frac{q}{2}\|_2^{\frac{2}{q}} \leq K_2(\mu-\mu_0)^{\frac{1}{q}-\frac{1}{2}}\|f\|_q, \quad \mu>\mu_0,
\end{equation*}
for some constants $\mu_0 \equiv \mu_0(d,q,\delta)>0$ and $K_i=K_i(d,q,\delta)>0$, $i=1,2$.
In particular, if $\delta<1 \wedge \big(\frac{2}{d-2}\bigr)^2$, there exists $q>2 \vee (d-2)$ such that $u \in C^{0,\gamma}$, $\gamma=1-\frac{d-2}{q}$. The explicit dependence of the regularity properties of $u$ on $\delta$
(which effectively plays the role of a ``coupling constant'') is a crucial feature of the result in \cite{KS}.

In the present paper our concern is: to find a class of matrices $a \in (H_u)$ such that the operator \eqref{div0}  with $b \in \mathbf{F}_\delta$ admits a $W^{1,p}$ and $C^{0,\gamma}$ regularity theory.
Below we consider
\begin{equation}
\label{a}
\tag{$\star$}
a=I+c\,\mathsf{f} \otimes \mathsf{f}, \quad c>-1, \quad \mathsf{f} \in \bigl[L^\infty \cap W_{\loc}^{1,2}\bigr]^d, \quad \|\mathsf{f}\|_\infty=1,
\end{equation}
\begin{equation}
\label{C}
\tag{$\mathbf{C}_{\delta_{\mathsf{f}}}$}
\begin{array}{c}
\nabla_i \mathsf{f} \in \mathbf{F}_{\delta^i},\;\;\delta^i>0,\;\;i=1,2,\dots,d, \quad
 \delta_{\mathsf{f}}:=\sum_{i=1}^d\delta^i.
\end{array}
\end{equation}
The model example of such $a$ is the matrix
\begin{equation}
\label{gs}
a(x)=I+c|x|^{-2}x\otimes x, \quad x \in \mathbb R^d
\end{equation}
having critical discontinuity at the origin, see \cite{GS2,GS,KiS3,OG} and references therein. (Replacing  the requirement $\nabla_i \mathsf{f} \in \mathbf{F}_{\delta^i}$ by a more restrictive $\nabla_i \mathsf{f} \in [L^p + L^\infty]^d$, $p>d$, forces $a$ to be H\"{o}lder continuous. 
On the other hand, a weaker condition $\nabla_i \mathsf{f} \in [L^p + L^\infty]^d$, $p<d$, is incompatible with the uniform ellipticity of $a$. 
The condition \eqref{C} ($\supsetneq \nabla_i \mathsf{f} \in [L^d + L^\infty]^d$) seems to be rather natural. We also note that the operator $-a\cdot\nabla^2$ with $\nabla_k a_{ij} \in L^{d,\infty}$ has been studied earlier  in \cite{AT}, cf.\,the discussion below concerning the non-divergence form operators.)

In Theorems \ref{thm:reg0}, \ref{thm:reg} below we characterize quantitative dependence on $c$, $\delta$, $\delta_{\mathsf{f}}$ of the $L^q \rightarrow W^{1,qd/(d-2)}$ regularity of the resolvent of the operator realization of \eqref{div0}  as (minus) generator of positivity preserving $L^\infty$ contraction $C_0$ semigroups in $L^q$, $q \geq 2 \vee ( d-2)$. 

Consider the non-divergence form operator
\begin{equation}
\label{nondiv0}
-a \cdot \nabla^2 + b \cdot \nabla \equiv -\sum_{i,j=1}^d a_{ij}(x) \nabla_i\nabla_j + \sum_{j=1}^d b_j(x) \nabla_j.
\end{equation}
Write
$
-a \cdot \nabla^2 + b \cdot \nabla \equiv -\nabla \cdot a \cdot \nabla + (\nabla a + b )\cdot \nabla,
$
where $(\nabla a)_k := \sum_{i=1}^d (\nabla_i a_{ik})$, $k=1,2,\dots,d$. Then
$$\nabla a=c\big[({\rm div}\mathsf{f})\mathsf{f} + \mathsf{f} \cdot \nabla\mathsf{f}\big].$$
It is easily seen that the condition \eqref{C} yields $\nabla a \in \mathbf{F}_{\delta_a}$ with $\delta_a \leq |c|^2(\sqrt{d}+1)^2\delta_{\mathsf{f}}$. The latter yields an analogue of Theorem \ref{thm:reg} for \eqref{nondiv0} (Corollary \ref{thm:reg2} below).

Theorem \ref{thm:reg} and Corollary \ref{thm:reg2} are needed to run an iteration procedure $L^p \rightarrow L^\infty$ that yields associated with \eqref{div0}, \eqref{nondiv0} Feller semigroups on $C_\infty=C_\infty(\mathbb R^d)$ (the space of all continuous functions vanishing at infinity endowed with the $\sup$-norm), see Theorem \ref{thm:feller} and Corollary \ref{thm:feller2} below.

In the same manner as it was done in \cite{KiS2} for the operator $-\Delta + b \cdot \nabla$, the Feller process constructed in Corollary \ref{thm:feller2} admits a characterization as a weak solution to the stochastic differential equation
$$
dX_t=-b(X_t)dt+\sqrt{2a(X_t)}dW_t, \quad X_0=x_0 \in \mathbb R^d.
$$
We plan to address this matter in another paper.

All the proofs below work for
\begin{equation}
\label{C2}
a=I+\sum_{j=1}^\infty c_j \mathsf{f}_j \otimes \mathsf{f}_j, \quad \|\mathsf{f}_j\|_\infty=1,
\end{equation}
with $\mathsf{f}_j$ satisfying \eqref{C}, and $c_+:=\sum_{c_j>0}c_j<\infty$, $c_-:=\sum_{c_j<0}c_j>-1$.
(A decomposition \eqref{C2} can be obtained from the spectral decomposition of a general uniformly elliptic $a$.)

\medskip

\textbf{2.~}We now state our results in full.

\begin{theorem}[$-\nabla \cdot a \cdot \nabla$]
\label{thm:reg0}
Let $d \geq 3$.
Let $a=I+c\,\mathsf{f} \otimes \mathsf{f}$ be given by \eqref{a}. 

\smallskip

{\rm(\textit{i})}  The formal differential expression $-\nabla \cdot a \cdot \nabla$ has an operator realization  $A_q$ in $L^q$ for all $q \in \big[1,\infty[$ 
as the (minus) generator of a  positivity preserving $L^\infty$ contraction $C_0$ semigroup.   

\smallskip

{\rm(\textit{ii})} 
Assume that \eqref{C} holds with $\delta_{\mathsf{f}}$, $c$ and $q \geq 2 \vee (d-2)$ satisfying the following constraint:
$$
-\bigl(1+q\sqrt{\delta_{\mathsf{f}}}\bigr)^{-1}<c<
\left\{
\begin{array}{ll}
16\bigl[q\sqrt{\delta_{\mathsf{f}}}\bigl(8+q\sqrt{\delta_{\mathsf{f}}} \bigr)\bigr]^{-1} & \text{ if } q\sqrt{\delta_{\mathsf{f}}}\leq 4, \\
\bigl(q\sqrt{\delta_{\mathsf{f}}}-1 \bigr)^{-1} & \text{ if } q\sqrt{\delta_{\mathsf{f}}}\geq 4.
\end{array}
\right.
$$
Then, for each $\mu>0$ and $f \in L^q$, $u:=(\mu+A_q)^{-1}f$ belongs to $W^{1,q} \cap W^{1,\frac{qd}{d-2}}$. Moreover, there exist constants 
$\mu_0=\mu_0(d,q,c,\delta_{\mathsf{f}})>0$ and $K_l=K_l(d,q,c,\delta_{\mathsf{f}})$, $l=1,2$, such that, for all $\mu>\mu_0$, 
\begin{equation}
\tag{$\star\star
$}
\label{reg_double_star}
\begin{array}{c}
\|\nabla u \|_{q} \leq  K_1(\mu-\mu_0)^{-\frac{1}{2}} \|f\|_q, \\[2mm]
\|\nabla u \|_{\frac{qd}{d-2}} \leq  K_2 (\mu-\mu_0)^{\frac{1}{q}-\frac{1}{2}} \|f\|_q.
\end{array}
\end{equation}
\end{theorem}

\begin{remarks}
1.~$\delta_{\mathsf{f}}$ effectively estimates from above the ``size'' of the discontinuities of $a$.

2.~For the matrix \eqref{gs}, the constraints on $c$ in Theorem \ref{thm:reg0} (and in other results below) can be substantially relaxed, see \cite{KiS3}.
\end{remarks}

\begin{theorem}[$-\nabla \cdot a \cdot \nabla + b \cdot \nabla$]
\label{thm:reg}
Let $d \geq 3$.
Let $a=I+c\,\mathsf{f} \otimes \mathsf{f}$ be given by \eqref{a}. Let $b \in \mathbf F_{\delta}$. 

\smallskip

{\rm(\textit{i})}  If $\delta_1:=[1 \vee (1+c)^{-2}]\,\delta < 4$, then $-\nabla \cdot a \cdot \nabla + b \cdot \nabla$ has an operator realization  $\Lambda_q(a,b)$ in $L^q$ for all $q \in \big[\frac{2}{2-\sqrt{\delta_1}},\infty[$
as the (minus) generator of a  positivity preserving $L^\infty$ contraction $C_0$ semigroup.

{\rm(\textit{ii})} Assume that \eqref{C} holds, $\nabla a \in \mathbf{F}_{\delta_a}$, $\delta< 1 \wedge \big(\frac{2}{d-2}\big)^2$, $\delta_a$, $\delta_{\mathsf{f}}$, $c$ and $q \geq 2 \vee (d-2)$ satisfy the constraints: 
\begin{align*}
0<c<(q-1-Q)\left\{
\begin{array}{ll}
 \bigl[(q-1)\frac{q\sqrt{\delta_{\mathsf{f}}}}{2} + \frac{q^2(\sqrt{\delta_{\mathsf{f}}}+ \sqrt{\delta})^2}{16}+ (q-2)\frac{q^2\delta_{\mathsf{f}}}{16}\bigr]^{-1}  &   \text{ if\;\;  $1-\frac{q\sqrt{\delta_{\mathsf{f}}}}{4}-\frac{q\sqrt{\delta}}{4} \geq 0$}, \\[4mm]
 \bigl(\frac{q^2\sqrt{\delta_{\mathsf{f}}}}{2}  +(q-2)\frac{q^2\delta_{\mathsf{f}}}{16} + \frac{q\sqrt{\delta}}{2}-1\bigr)^{-1}   & \text{ if\;\; $0 \leq 1-\frac{q\sqrt{\delta_{\mathsf{f}}}}{4} < \frac{q\sqrt{\delta}}{4}$}, \\[4mm]
\bigl[(q-1)\bigl(q\sqrt{\delta_{\mathsf{f}}}  -1\bigr) + \frac{q\sqrt{\delta}}{2}\bigr]^{-1} & \text{ if\;\;  $1-\frac{q\sqrt{\delta_{\mathsf{f}}}}{4} < 0$},
\end{array}
\right.
\end{align*}
where $Q:=\frac{q\sqrt{\delta}}{2}\bigl[q-2+\big(\sqrt{\delta_a}+\sqrt{\delta}\big)\frac{q}{2}\bigr]$, or
$$
-\bigl(q-1 - Q\bigr)\biggl[ (q-1)\bigl(1+q\sqrt{\delta_{\mathsf{f}}}) + \frac{q\sqrt{\delta}}{2}\biggr]^{-1}<c<0.
$$
Then there exist constants 
$\mu_0=\mu_0(d,q,c,\delta,\delta_a,\delta_{\mathsf{f}})>0$ and $K_l=K_l(d,q,c,\delta,\delta_a,\delta_{\mathsf{f}})$, $l=1,2$, such that
 \eqref{reg_double_star} hold for $u:=(\mu+\Lambda_q(a,b))^{-1}f$, $\mu>\mu_0$, $f \in L^q$.

\end{theorem}

\begin{remarks}
1.~Taking $c=0$ (then $\delta_a=0$), we recover in Theorem \ref{thm:reg}(\textit{ii}) the result of \cite[Lemma 5]{KS}: $\delta< 1 \wedge \big(\frac{2}{d-2}\big)^2$.

2.~Theorem \ref{thm:reg}(\textit{i}) is an immediate consequence of the following general result. 
Let $a$ be an  $\mathcal L^d$ measurable uniformly elliptic matrix on $\mathbb R^d$. 
Set $A \equiv A_2:= [-\nabla \cdot a \cdot \nabla\upharpoonright C_c^\infty]^{\rm clos}_{2 \to 2}$.
A vector field $b:\mathbb R^d \rightarrow \mathbb R^d$ belongs to $\mathbf F_{\delta_1}(A)$, $\delta_1>0$, the class of form-bounded vector fields (with respect to $A$), if $b_a^2:=b \cdot a^{-1} \cdot b
 \in L^1_\loc$ and there exists a constant $\lambda=\lambda_{\delta_1}>0$ such that 
$$\|b_a(\lambda + A)^{-\frac{1}{2}}\|_{2\rightarrow 2}\leq \sqrt{\delta_1}.$$
If $b \in \mathbf{F}_{\delta_1}(A)$, $\delta_1<4$, then $-\nabla \cdot a \cdot \nabla + b \cdot \nabla$ has an operator realization  $\Lambda_q(a,b)$ in $L^q$ for all $q \in \big[\frac{2}{2-\sqrt{\delta_1}},\infty[$
as the (minus) generator of a  positivity preserving $L^\infty$ contraction $C_0$ semigroup, see \cite[Theorem 3.2]{KiS}.  
\end{remarks}

\begin{corollary}[$-a \cdot \nabla^2 + b \cdot \nabla$]
\label{thm:reg2}
Let $d \geq 3$.
Let $a=I+c\,\mathsf{f} \otimes \mathsf{f}$ be given by \eqref{a}. Let $b \in \mathbf F_{\delta}$, $\nabla a \in \mathbf{F}_{\delta_a}$. Then $\nabla a + b \in \mathbf{F}_{\delta_2}$, $\sqrt{\delta_2}:=\sqrt{\delta_a} + \sqrt{\delta}$.

\smallskip

{\rm(\textit{i})}  If $\delta_1:=[1 \vee (1+c)^{-2}]\,\delta_2<4$, then $-a \cdot \nabla^2 + b \cdot \nabla$ has an operator realization  $\Lambda_q(a,\nabla a + b)$ in $L^q$ for all $q \in \big[\frac{2}{2-\sqrt{\delta_1}},\infty[$
as the (minus) generator of a  positivity preserving $L^\infty$ contraction $C_0$ semigroup.

{\rm(\textit{ii})} Assume that \eqref{C} holds, and $\delta_2< 1 \wedge \big(\frac{2}{d-2}\big)^2$, $\delta_a$, $\delta_{\mathsf{f}}$, $c$, $q \geq 2 \vee (d-2)$ satisfy the constraints:
\begin{align*}
0<c<(q-1-Q)\left\{
\begin{array}{ll}
 \bigl[(q-1)\frac{q\sqrt{\delta_{\mathsf{f}}}}{2} + \frac{q^2(\sqrt{\delta_{\mathsf{f}}}+ \sqrt{\delta_2})^2}{16}+ (q-2)\frac{q^2\delta_{\mathsf{f}}}{16}\bigr]^{-1}  &   \text{ if\;\;  $1-\frac{q\sqrt{\delta_{\mathsf{f}}}}{4}-\frac{q\sqrt{\delta_2}}{4} \geq 0$}, \\[4mm]
 \bigl(\frac{q^2\sqrt{\delta_{\mathsf{f}}}}{2}  +(q-2)\frac{q^2\delta_{\mathsf{f}}}{16} + \frac{q\sqrt{\delta_2}}{2}-1\bigr)^{-1}   & \text{ if\;\; $0 \leq 1-\frac{q\sqrt{\delta_{\mathsf{f}}}}{4} < \frac{q\sqrt{\delta_2}}{4}$}, \\[4mm]
\bigl[(q-1)\bigl(q\sqrt{\delta_{\mathsf{f}}}  -1\bigr) + \frac{q\sqrt{\delta_2}}{2}\bigr]^{-1} & \text{ if\;\;  $1-\frac{q\sqrt{\delta_{\mathsf{f}}}}{4} < 0$},
\end{array}
\right.
\end{align*}
where $Q:=\frac{q\sqrt{\delta_2}}{2}\bigl[q-2+\big(\sqrt{\delta_a}+\sqrt{\delta_2}\big)\frac{q}{2}\bigr]$, or
$$
-\bigl(q-1 - Q\bigr)\biggl[ (q-1)\bigl(1+q\sqrt{\delta_{\mathsf{f}}}) + \frac{q\sqrt{\delta_2}}{2}\biggr]^{-1}<c<0.
$$
Then there exist constants 
$\mu_0=\mu_0(d,q,c,\delta_2,\delta_a,\delta_{\mathsf{f}})>0$ and $K_l=K_l(d,q,c,\delta_2,\delta_a,\delta_{\mathsf{f}})$, $l=1,2$, such that the estimates \eqref{reg_double_star} hold for $u=(\mu+\Lambda_q(a,\nabla a + b))^{-1}f$, $\mu>\mu_0$, $f \in L^q$.
\end{corollary}

Set $b_n := e^{\epsilon_n\Delta} (\mathbf 1_n b)$,  $\epsilon_n \downarrow 0$, $n \geq 1$, where $\mathbf 1_n$ is the indicator of  $\{x \in \mathbb R^d \mid  \; |x| \leq n,  |b(x)| \leq n \}$.
Also, set $\mathsf{f}_n:=(\mathsf{f}_{n}^i)_{i=1}^d$, 
$
\mathsf{f}_{n}^i:=e^{\epsilon_n\Delta}(\eta_n\mathsf{f}^i),
$ 
$\epsilon_n \downarrow 0$, $n \geq 1$, where 
$$
\eta_n(x):=\left\{
\begin{array}{ll}
1, & \text{ if } |x|< n, \\
n+1-|x|, & \text{ if } n \leq |x| \leq n+1, \qquad (x \in \mathbb R^d)\\
0, & \text{ if } |x|>n+1.
\end{array}
\right.
$$

\begin{theorem}[$-\nabla \cdot a \cdot \nabla + b \cdot \nabla$]
\label{thm:feller}
{\rm(\textit{i})} In the assumptions of Theorem \ref{thm:reg}(\textit{ii}), the formal differential operator $-\nabla \cdot a \cdot \nabla + b\cdot \nabla$ has an operator realization $-\Lambda_{C_\infty}(a,b)$ as the generator of a positivity preserving contraction $C_0$ semigroup in $C_\infty$ defined by
\[
e^{-t\Lambda_{C_\infty}(a,b)}:=s \mbox{-} C_\infty \mbox{-} \lim_n e^{-t \Lambda_{C_\infty}(a_n,b_n)} \quad (\text{loc.\;uniformly in }t \geq 0),
\] 
where $a_n:=I+c\,\mathsf{f}_n \otimes \mathsf{f}_n \subset [C^\infty]^{d \times d}$,
$\Lambda_{C_\infty}(a_n,b_n):=-\nabla \cdot a_n \cdot \nabla +b_n\cdot \nabla$, 
$D(\Lambda_{C_\infty}(a_n,b_n))=(1-\Delta)^{-1}C_\infty$.

\smallskip

{\rm(\textit{ii})}~{\rm[}The $L^r$-strong Feller property {\rm]}~$\bigl( (\mu+\Lambda_{C_\infty}(a,b))^{-1} \upharpoonright L^r \cap C_\infty\bigr)^{\clos}_{L^r \rightarrow C_\infty} \in \mathcal B(L^r,  C^{0,1- \frac{d}{rj}})$ for some $r>d-2$ and all $\mu > \mu_0$. 

\smallskip

{\rm(\textit{iii})}~The
integral kernel of $e^{-t\Lambda_{C_\infty}(a,b)}$ determines the
transition probability function of a Feller process.
\end{theorem}

\begin{corollary}[$-a \cdot \nabla^2 + b \cdot \nabla$]
\label{thm:feller2}
{\rm(\textit{i})} In the assumptions of Corollary \ref{thm:reg2}(\textit{ii}), the formal differential operator $-a \cdot \nabla^2 + b\cdot \nabla$ has an operator realization $-\Lambda_{C_\infty}(a,\nabla a + b)$ as the generator of a positivity preserving contraction $C_0$ semigroup in $C_\infty$ defined by
\[
e^{-t\Lambda_{C_\infty}(a,\nabla a + b)}:=s \mbox{-} C_\infty \mbox{-} \lim_n e^{-t \Lambda_{C_\infty}(a_n,\nabla a_n + b_n)} \quad (\text{loc.\;uniformly in }t \geq 0),
\] 
where $a_n=I+c\,\mathsf{f}_n \otimes \mathsf{f}_n \subset [C^\infty]^{d \times d}$, 
$\Lambda_{C_\infty}(a_n,\nabla a_n+b_n):=-a_n \cdot \nabla^2 +b_n\cdot \nabla$, 
$D(\Lambda_{C_\infty}(a_n,\nabla a_n+ b_n))=(1-\Delta)^{-1}C_\infty$.

\smallskip

{\rm(\textit{ii})}~{\rm[}The $L^r$-strong Feller property {\rm]}~$\bigl( (\mu+\Lambda_{C_\infty}(a,\nabla a + b))^{-1} \upharpoonright L^r \cap C_\infty\bigr)^{\clos}_{L^r \rightarrow C_\infty} \in \mathcal B(L^r,  C^{0,1- \frac{d}{rj}})$ for some $r>d-2$ and all $\mu > \mu_0$. 

\smallskip

{\rm(\textit{iii})}~The
integral kernel of $e^{-t\Lambda_{C_\infty}(a,\nabla a + b)}$ determines the
transition probability function of a Feller process.
\end{corollary}

\begin{remarks}
Since our assumptions on $\delta_{\mathsf{f}}$, $\delta_a$ and $\delta$ involve only strict inequalities, we may assume that 
\begin{equation}
\label{incl}
\text{\eqref{C} holds for $\mathsf{f}_n$}, \quad \nabla a_n \in \mathbf{F}_{\delta_a}, \quad b_n  \in \mathbf{F}_{\delta} \quad  \text{ with $\lambda \neq \lambda(n)$}
\end{equation}
for appropriate $\epsilon_n \downarrow 0$. 
In fact, the proofs work for any approximations $\{\mathsf{f}_n\}$, $\{b_n\} \subset [C^\infty]^d$ such that $\|\mathsf{f}_n\|_\infty=1$, \eqref{incl} holds, and
$$
\mathsf{f}_n \rightarrow \mathsf{f}, \nabla_i\mathsf{f}_n \rightarrow \nabla_i\mathsf{f} \text{ strongly in } [L_{\loc}^2]^{d}, \quad i=1,2,\dots,d,
$$
$$
b_n \rightarrow b \text{ strongly in } [L_{\loc}^2]^{d}.
$$
\end{remarks}

\section{Proof of Theorem \ref{thm:reg0}}

\textit{Proof of (i)}. In what follows, we use notation
$$
\langle h\rangle:=\int_{\mathbb R^d} h(x)dx, \quad \langle h,g\rangle:=\langle h\bar{g}\rangle.
$$

Define $t[u,v]:=\langle \nabla u \cdot a \cdot \nabla \bar{v} \rangle$, $D(t)=W^{1,2}$.
There is a unique self-adjoint operator $A\equiv A_2 \geq 0$ on $L^2$ associated with the form $t$: $D(A) \subset D(t)$, $\langle Au,v\rangle=t[u,v]$, $u \in D(A)$, $v \in D(t)$. $-A$ is the generator of a  positivity preserving $L^\infty$ contraction $C_0$ semigroup $T^t_2 \equiv e^{-tA}$,  $t \geq 0$, on $L^2$. Then
$T_r^t:=[T_t\upharpoonright L^r \cap L^2]_{L^r \rightarrow L^r}$ determines $C_0$ semigroup on $L^r$ for all $r \in [1,\infty[$.
The  generator $-A_r$ of $T^t_r \,(\equiv e^{-tA_r})$ is the desired operator realization of $\nabla \cdot a \cdot \nabla$ in $L^r$, $r \in [1,\infty[$. 
Moreover, $(\mu+A_r)^{-1}$ is well defined on $L^r$ for all $\mu>0$. This completes the proof of the  assertion (\textit{i}) of the theorem.

\smallskip

\textit{Proof of (ii)}. First, we prove an a priori variant of \eqref{reg_double_star}. 
Set $a_n:=I+c\mathsf{f}_n \otimes \mathsf{f}_n$, where $\mathsf{f}_n$ have been defined in the beginning of the paper.
Since our assumption on $\delta_{\mathsf{f}}$ is a strict inequality, we may assume that \eqref{C} holds for $\mathsf{f}_n$ for all $n \geq 1$ with $\lambda \neq \lambda(n)$ for appropriate $\epsilon_n \downarrow 0$.
We also note that $\|\mathsf{f}_n\|_\infty=1$. 

Set $u \equiv u_n := (\mu + A^n_q)^{-1} f$, $0 \leq f \in C_c^1$, where $A_q^n:= -\nabla \cdot a_n \cdot \nabla$, $D(A^n_q)=W^{2,q}$, $n \geq 1$. Clearly, $ 0 \leq u_n \in W^{3,q}$.

Denote $w \equiv w_n:=\nabla u_n$. For brevity, 
below we omit the index $n$: $\mathsf{f} \equiv \mathsf{f}_n$, $a \equiv a_n$, $A_q \equiv A_q^n$.
 Set
$$
I_q:=\sum_{r=1}^d\langle (\nabla_r w)^2 |w|^{q-2} \rangle, \quad
J_q:=\langle (\nabla |w|)^2 |w|^{q-2}\rangle,$$
$$
\bar{I}_{q}:= \langle  \bigl( \mathsf{f} \cdot \nabla w \bigr)^2 |w|^{q-2}\rangle, \quad
\bar{J}_{q}:= \langle (\mathsf{f} \cdot \nabla |w|)^2  |w|^{q-2}\rangle.
$$
Set $[F,G]_-:=FG-GF$.

\textbf{1.~}We multiply $\mu u  + A_q u = f $ by $\phi^{}:=- \nabla \cdot (w |w|^{q-2})$ and integrate:
\begin{equation*}
\mu \langle |w|^q \rangle +\langle A_q w^{}, w|w|^{q-2} \rangle + \langle [\nabla,A_q]_-u^{}, w|w|^{q-2}\rangle = \langle f, \phi^{} \rangle,
\end{equation*}
\[
\mu \langle |w|^q \rangle + I_q^{} + c\bar{I}_{q}^{} + (q-2)(J_q + c\bar{J}_{q}^{})+ \langle [\nabla,A_q]_-u^{}, w|w|^{q-2}\rangle = \langle  f, \phi \rangle.
\]
The term to evaluate is this: $$\langle [\nabla,A_q]_- u^{},w|w|^{q-2}\rangle :=\sum_{r=1}^d\langle [\nabla_r,A_q]_- u^{},w_r|w|^{q-2}\rangle.$$ From now on, we omit the summation sign in repeated indices.
Note that 
$$[\nabla_r,A_q]_-=-\nabla \cdot (\nabla_r a) \cdot \nabla, \qquad (\nabla_r a)_{il} = c(\nabla_r \mathsf{f}^i)\mathsf{f}^l + c\mathsf{f}^i\nabla_r \mathsf{f}^l.$$
Thus,
\begin{align*}
\langle [\nabla_r,A_q]_- u^{},w_r|w|^{q-2}\rangle=c\left\langle \bigl[(\nabla_r \mathsf{f}^i)\mathsf{f}^l + \mathsf{f}^i\nabla_r \mathsf{f}^l \bigr]w_l,\nabla_i(w_r|w|^{q-2})\right\rangle =:S_{1}+S_{2},
\end{align*}
\begin{align*}
S_{1}=c \big\langle (\nabla_r \mathsf{f}) \cdot (\nabla_r w) (\mathsf{f} \cdot w)|w|^{q-2}\big\rangle & + c(q-2)\big\langle (\nabla_r\mathsf{f}) \cdot (\nabla|w|) (\mathsf{f} \cdot w) w_r|w|^{q-3}\big\rangle, 
\end{align*}
\begin{align*}
S_{2}=c \big\langle (\nabla_r\mathsf{f}) \cdot w, (\mathsf{f} \cdot \nabla w_r)|w|^{q-2}\big\rangle &+ c(q-2)\big\langle (\nabla_r\mathsf{f}) \cdot w, w_r|w|^{q-3}\mathsf{f} \cdot \nabla|w| \big\rangle.
\end{align*}
By the quadratic estimates and the condition \eqref{C},
$$
S_{1} \leq |c| \biggl[\alpha\bigg(\delta_{\mathsf{f}} \frac{q^2}{4}J_q + \lambda\delta_{\mathsf{f}} \|w\|_q^q \bigg) + \frac{1}{4\alpha}I_q\biggr] + |c|(q-2) \biggl[\alpha_1\bigg(\delta_{\mathsf{f}} \frac{q^2}{4}J_q + \lambda\delta_{\mathsf{f}} \|w\|_q^q \bigg) + \frac{1}{4\alpha_1} J_q\biggr], \quad \alpha, \alpha_1>0
$$
$$
S_{2} \leq |c| \biggl[\gamma\bigg(\delta_{\mathsf{f}} \frac{q^2}{4}J_q + \lambda\delta_{\mathsf{f}} \|w\|_q^q \bigg) + \frac{1}{4\gamma}\bar{I}_q\biggr] + |c|(q-2) \biggl[\gamma_1\bigg(\delta_{\mathsf{f}} \frac{q^2}{4}J_q + \lambda\delta_{\mathsf{f}} \|w\|_q^q \bigg) + \frac{1}{4\gamma_1}\bar{J}_q\biggr], \quad \gamma, \gamma_1>0.
$$
Thus, selecting $\alpha = \alpha_1 =\frac{1}{q\sqrt{\delta_{\mathsf{f}}}}$, we obtain the inequality
\begin{align}
&\mu \|w\|_q^q + I_q + c\bar{I}_{q} + (q-2)(J_q + c\bar{J}_q) \notag \\ 
&\leq |c| \biggl[q\frac{\sqrt{\delta_{\mathsf{f}}}}{4}J_q + \frac{q\sqrt{\delta_{\mathsf{f}}}}{4}I_q\biggr] + |c|(q-2) \frac{q\sqrt{\delta_{\mathsf{f}}}}{2}J_q \notag \\
&+ |c| \biggl[\gamma\delta_{\mathsf{f}} \frac{q^2}{4}J_q + \frac{1}{4\gamma}\bar{I}_q\biggr] + |c|(q-2) \bigg[\gamma_1 \delta_{\mathsf{f}} \frac{q^2}{4}J_q +\frac{1}{4\gamma_1}\bar J_q \bigg] \label{be2_}\\
&+\mu_{0}\|w\|_q^q + \langle f,\phi \rangle \notag 
\end{align}
where $\mu_{0}:=|c|\lambda\sqrt{\delta_{\mathsf{f}}}\big(q^{-1} + \gamma \sqrt{\delta_{\mathsf{f}}})+|c|(q-2)\lambda\sqrt{\delta_{\mathsf{f}}}\big(q^{-1} + \gamma_1 \sqrt{\delta_{\mathsf{f}}} \big)$. 
\smallskip

\smallskip

\textbf{2.~}Let us prove that there exists constant $\eta>0$ such that
\begin{equation}
\label{star_ineq}
\tag{$\ast$}
(\mu - \mu_0) \| w \|_q^q + \eta J_q \leq \langle f,\phi \rangle.
\end{equation}

\textbf{Case $c>0$}. First, suppose that $1-\frac{q\sqrt{\delta_{\mathsf{f}}}}{4} \geq 0$. We select $\gamma =\gamma_1:=\frac{1}{4}$, 
so the terms $\bar{I}_q$, $\bar{J}_q$ are no longer present in \eqref{be2_}. By the assumption of the theorem $1-c\frac{q\sqrt{\delta_{\mathsf{f}}}}{4} \geq 0$, so using $J_q \leq I_q$ we obtain 
\begin{align*}
(\mu - \mu_{0}) \|w\|_q^q + (q-1)\bigg[1- c\frac{q\sqrt{\delta_{\mathsf{f}}}}{2} - c\frac{q^2\delta_{\mathsf{f}}}{16}\bigg]J_q  
\leq \langle f, \phi \rangle, 
\end{align*}
where $\mu_{0}= c \lambda\sqrt{\delta_{\mathsf{f}}}(q-1)\big(\frac{1}{q} + \frac{\sqrt{\delta_{\mathsf{f}}}}{4}\big)$ and the coefficient $[\dots]$ is strictly positive by the assumptions of the theorem.

Now, suppose that $1-\frac{q\sqrt{\delta_{\mathsf{f}}}}{4} < 0$. We select $\gamma=\gamma_1:=\frac{1}{q\sqrt{\delta_{\mathsf{f}}}}$ and replace $\bar{J}_q$, $\bar{I}_q$ by $J_q$, $I_q$. Then, since $1-c\big(\frac{q\sqrt{\delta_{\mathsf{f}}}}{2}-1\big) \geq 0$ by the assumptions of the theorem, we apply $J_q \leq I_q$ to obtain
\begin{align*}
(\mu - \mu_{0}) \|w\|_q^q + (q-1) \bigg[1-c\big(q\sqrt{\delta_{\mathsf{f}}}-1\big) \bigg]J_q   \leq \langle f, \phi \rangle,
\end{align*}
where $\mu_{0} = c \lambda\sqrt{\delta_{\mathsf{f}}}(q-1)\big(\frac{1}{q} + \frac{1}{q}\big)$ and the coefficient  $[\dots]$ is strictly positive by the assumption of the theorem.
We have proved \eqref{star_ineq} with $\mu_0 = c \lambda\sqrt{\delta_{\mathsf{f}}}(q-1)\big(\frac{1}{q} + \frac{\sqrt{\delta_{\mathsf{f}}}}{4} \vee \frac{1}{q}\big)$.

\begin{remark*}
Elementary considerations show that the above choice of $\alpha$, $\alpha_1$, $\gamma$, $\gamma_1$ is the best possible.
\end{remark*}

\smallskip

\textbf{Case $c<0$}. We select $\gamma=\gamma_1:=\frac{1}{q\sqrt{\delta_{\mathsf{f}}}}$, so that
\begin{align*}
&\mu \|w\|_q^q + \bigg(1-|c|\frac{q\sqrt{\delta_{\mathsf{f}}}}{4}\bigg)I_q^{}   + \biggl[q-2- |c|(q-1)\frac{q\sqrt{\delta_{\mathsf{f}}}}{2}-|c|(q-2)\frac{q\sqrt{\delta_{\mathsf{f}}}}{4} \biggr]J_q \notag   \\
& \leq |c|\bigg(1+\frac{q\sqrt{\delta_{\mathsf{f}}}}{4}\bigg)\bar{I}_{q}^{} + |c|(q-2)\bigg(1+\frac{q\sqrt{\delta_{\mathsf{f}}}}{4}\bigg)\bar{J}_{q}^{} + \mu_0\|w\|_q^{q}+ \langle f, \phi \rangle,
\end{align*}
where $\mu_0=2c \lambda\sqrt{\delta_{\mathsf{f}}}\frac{q-1}{q}$. 
Next, using $\bar I_q \leq I_q$, $\bar J_q \leq J_q$, we obtain
\begin{align*}
&(\mu - \mu_0) \|w\|_q^q + \bigg(1- |c|-|c|\frac{q\sqrt{\delta_{\mathsf{f}}}}{2}\bigg) I_q^{}  \\
& + \biggl[q-2- |c|(q-1)\frac{q\sqrt{\delta_{\mathsf{f}}}}{2}-|c|(q-2)\frac{q\sqrt{\delta_{\mathsf{f}}}}{4} -|c|(q-2)\bigg(1+\frac{q\sqrt{\delta_{\mathsf{f}}}}{4}\bigg) \biggr]J_q \notag  \leq \langle f, \phi \rangle.
\end{align*}
By the assumptions of the theorem, $1- |c|-|c|\frac{q\sqrt{\delta_{\mathsf{f}}}}{2} \geq 0$. Therefore, by $I_q \geq J_q$,
\[
(\mu - \mu_0) \|w\|_q^q + \biggl[q-1 - |c|(q-1)- |c|q^2 \sqrt{\delta_{\mathsf{f}}} \biggr]J_q \notag  \leq \langle f, \phi \rangle,
\]
and hence the coefficient $[\dots]$ is strictly positive. We have proved \eqref{star_ineq}.

\smallskip

\textbf{3.~}We estimate the term $\langle f, \phi \rangle$  as follows.

\begin{lemma} 
\label{f_est_lem}
For each $\varepsilon_0>0$ there exists a constant $C=C(\varepsilon_0)<\infty$  such that 
$$
\langle f, \phi \rangle \leq \varepsilon_0 I_q + C \|w\|_q^{q-2} \|f\|^2_q.
$$
\end{lemma}

\begin{proof}[Proof of Lemma \ref{f_est_lem}]
We have:
\[
 \langle f, \phi \rangle =\langle - \Delta u, |w|^{q-2}f\rangle + (q-2) \langle |w|^{q-3} w \cdot \nabla |w|,f\rangle=:F_1+F_2.
 \]
Due to $|\Delta u|^2 \leq d |\nabla_r w|^2$ and $\langle |w|^{q-2}f\rangle \leq \|w\|_q^{q-2}\|f\|^2_q$,
\[
F_1 \leq \sqrt{d}I_q^{\frac{1}{2}}\|w\|_q^{\frac{q-2}{2}}\|f\|_q, \qquad F_2 \leq (q-2)J_q^{\frac{1}{2}}\|w\|_q^{\frac{q-2}{2}}\|f\|_q,  
\]
Now the standard quadratic estimates yield the lemma.
\end{proof}

We choose $\varepsilon_0>0$ in Lemma \ref{f_est_lem} so small that in the estimates below we can ignore 
$\varepsilon_0 I_q$.

\textbf{4.~}Clearly, \eqref{star_ineq} yields the inequalities
$$ \|\nabla u_n\|_{q} \leq  K_1(\mu-\mu_0)^{-\frac{1}{2}} \|f\|_q, \quad K_1:=C^\frac{1}{2},$$
$$\|\nabla u_n \|_{q j} \leq  K_2(\mu-\mu_0)^{-\frac{1}{2}+\frac{1}{q}} \|f\|_q, \quad K_2:=C_{S}\eta^{-\frac{1}{q}}(q^2/4)^{\frac{1}{q}} C^{\frac{1}{2}-\frac{1}{q}},$$
where $C_S$ is the constant in the Sobolev Embedding Theorem.
So, \cite[Theorem 3.5]{KiS}
($
(\mu+A_q)^{-1}=s{\mbox-}L^q{\mbox-}\lim_n (\mu+A_q^n)^{-1}$)
yields \eqref{reg_double_star}.
The proof of Theorem \ref{thm:reg0} is completed.

\section{Proof of Theorem \ref{thm:reg}}

\textit{Proof of (i)}. Recall that a vector field $b:\mathbb R^d \rightarrow \mathbb R^d$ belongs to $\mathbf F_{\delta_1}(A)$, $\delta_1>0$, the class of form-bounded vector fields (with respect to $A \equiv A_2:= [-\nabla \cdot a \cdot \nabla\upharpoonright C_c^\infty]^{\rm clos}_{2 \to 2}$), if $b_a^2:=b \cdot a^{-1} \cdot b
 \in L^1_\loc$ and there exists a constant $\lambda=\lambda_{\delta_1}>0$ such that 
$$\|b_a(\lambda + A)^{-\frac{1}{2}}\|_{2\rightarrow 2}\leq \sqrt{\delta_1}.$$
It is easily seen that if $b \in \mathbf{F}_\delta$, then 
$b \in \mathbf{F}_{\delta_1}(A)$, with $\delta_1:=[1 \vee (1+c)^{-2}]\,\delta$. By the assumptions of the theorem, $\delta_1<4$. 
Therefore, by \cite[Theorem 3.2]{KiS}, $-\nabla \cdot a \cdot \nabla + b \cdot \nabla$ has an operator realization  $\Lambda_q(a,b)$ in $L^q$, $q \in \big[\frac{2}{2-\sqrt{\delta_1}}, \infty\big[$, as the (minus) generator of a  positivity preserving $L^\infty$ contraction quasi contraction $C_0$ semigroup. Moreover, $(\mu+\Lambda_q(a,b))^{-1}$ is well defined on $L^q$ for all $\mu>\frac{\lambda \delta}{2(q-1)}$.
This completes the proof of (\textit{i}).

\smallskip

\textit{Proof of (ii)}.  First, we prove an a priori variant of \eqref{reg_double_star}. 
Set $a_n:=I+c\mathsf{f}_n \otimes \mathsf{f}_n$, where $\mathsf{f}_n$ have been defined in the beginning of the paper.
Since our assumptions on $\delta_{\mathsf{f}}$, $\delta_a$ and $\delta$ involve only strict inequalities, we may assume that \eqref{C} holds for $\mathsf{f}_n$, $\nabla a_n \in \mathbf{F}_{\delta_a}$, $b_n  \in \mathbf{F}_{\delta}$  with $\lambda \neq \lambda(n)$ for appropriate $\epsilon_n \downarrow 0$.
We also note that $\|\mathsf{f}_n\|_\infty=1$. 

Denote $A_q^n:= -\nabla \cdot a_n \cdot \nabla$, $D(A^n_q)=W^{2,q}$. Set $u \equiv u_n := (\mu + \Lambda_q(a_n,b_n))^{-1} f$, $0 \leq f \in C_c^1$, $n \geq 1$, where $\Lambda_q(a_n,b_n)=A^n_q + b_n \cdot \nabla$, $D(\Lambda_q(a_n,b_n))=D(A^n_q)$. Clearly, $ 0 \leq u_n \in W^{3,q}$. 
It is easily seen that $b_n \in \mathbf{F}_{\delta_1}(A^n)$ with $\lambda \neq \lambda(n)$, so $(\mu+\Lambda_q(a_n,b_n))^{-1}$ are well defined on $L^q$ for all $n \geq 1$, $\mu>\frac{\lambda \delta}{2(q-1)}$.

\smallskip

\textbf{1.~}Denote $w \equiv w_n:=\nabla u_n$. Below we omit the index $n$: $\mathsf{f} \equiv \mathsf{f}_n$, $a \equiv a_n$, $b \equiv b_n$,  $A_q \equiv A_q^n$.
 Set
$$
I_q:=\langle (\nabla_r w)^2 |w|^{q-2} \rangle, \quad
J_q:=\langle (\nabla |w|)^2 |w|^{q-2}\rangle,$$
$$
\bar{I}_{q}:= \langle  \bigl( \mathsf{f} \cdot \nabla w \bigr)^2 |w|^{q-2}\rangle, \quad
\bar{J}_{q}:= \langle (\mathsf{f} \cdot \nabla |w|)^2  |w|^{q-2}\rangle.
$$
\smallskip
Arguing as in the proof of Theorem \ref{thm:reg0}, we arrive at
\begin{align}
&\mu \langle |w|^q \rangle + I_q^{} + c\bar{I}_{q}^{} + (q-2)(J_q + c\bar{J}_{q}^{}) \notag \\ 
&\leq |c| \biggl[\alpha\delta_{\mathsf{f}} \frac{q^2}{4}J_q + \frac{1}{4\alpha}I_q\biggr] + |c|(q-2) \biggl[\alpha_1\delta_{\mathsf{f}} \frac{q^2}{4}J_q + \frac{1}{4\alpha_1}J_q\biggr] \label{be2} \\
&+ |c| \biggl[\gamma\delta_{\mathsf{f}} \frac{q^2}{4}J_q + \frac{1}{4\gamma}\bar{I}_q\biggr] + |c|(q-2) \biggl[\gamma_1\delta_{\mathsf{f}} \frac{q^2}{4}J_q + \frac{1}{4\gamma_1}\bar{J}_q\biggr] \notag \\
&+ \mu_{00}\|w\|_q^q + \langle -b \cdot w, \phi \rangle + \langle f,\varphi\rangle, \qquad \text{ with } \alpha=\alpha_1:=\frac{1}{q\sqrt{\delta_{\mathsf{f}}}}, \notag
\end{align}
where $\mu_{00}:=|c|\lambda\sqrt{\delta_{\mathsf{f}}}\big(q^{-1} + \gamma \sqrt{\delta_{\mathsf{f}}})+|c|(q-2)\lambda\sqrt{\delta_{\mathsf{f}}}\big(q^{-1} + \gamma_1 \sqrt{\delta_{\mathsf{f}}} \big)$, and $\gamma,\gamma_1>0$ are to be chosen.

\medskip

\textbf{2.~}We estimate the term $\langle -b\cdot w, \phi^{} \rangle$ as follows.

\begin{lemma}
\label{b_est_lem}
There exist constants $C_i$ \rm{($i=0,1$)} such that
\begin{align*} 
 \langle - b\cdot w, \phi \rangle \leq \biggl[\big(\sqrt{\delta}\sqrt{\delta_a}+\delta\big)\frac{q^2}{4}+(q-2)\frac{q\sqrt{\delta}}{2}\biggr]J_q + |c|\frac{q\sqrt{\delta}}{2}J_q^{\frac{1}{2}} \bar{I}_{q}^{\frac{1}{2}}
  + C_0\|w\|^q_q + C_1\|w\|_q^{q-2} \|f\|^2_q.
\end{align*}
\end{lemma}

\begin{proof}
We have:
\begin{align*}
 \langle-b\cdot w, \phi \rangle & =\langle - \Delta u, |w|^{q-2}(-b\cdot w)\rangle + (q-2) \langle |w|^{q-3} w \cdot \nabla |w|,-b\cdot w\rangle\\
&=:F_1+F_2.
\end{align*}
Set $B_q:=\langle |b \cdot w|^2|w|^{q-2}\rangle$. We have
$$F_2
\leq (q-2) B_{q}^\frac{1}{2} J_{q}^\frac{1}{2}.$$
Next, we bound $F_1$. Recall that $\nabla a=c\big[({\rm div}\mathsf{f})\mathsf{f} + \mathsf{f} \cdot \nabla\mathsf{f}\big]$. We represent $-\Delta u=\nabla \cdot (a-1)\cdot w -\mu u - b\cdot w +f$,
and evaluate: $\nabla \cdot (a-1)\cdot w=\nabla a \cdot w + c \mathsf{f} \cdot (\mathsf{f} \cdot \nabla w)$, so
\begin{align*}
F_1& =\langle \nabla \cdot (a-1)\cdot w, |w|^{q-2}(-b\cdot w)\rangle  + \langle (-\mu u - b\cdot w +f),|w|^{q-2}(-b\cdot w) \rangle \\
& = \langle \nabla a \cdot w, |w|^{q-2}(-b\cdot w) \rangle \\
& + c \langle \mathsf{f} \cdot (\mathsf{f} \cdot \nabla w), |w|^{q-2}(-b\cdot w) \rangle \\
& + \langle (-\mu u - b\cdot w +f),|w|^{q-2}(-b\cdot w) \rangle.
\end{align*}
Set $P_q:=\langle |\nabla a \cdot w|^2 |w|^{q-2} \rangle $. We bound $F_1$ from above by applying consecutively the following estimates:

\smallskip

$1^\circ$) $\langle \nabla a \cdot w, |w|^{q-2}(-b\cdot w) \rangle \leq P_{q}^{\frac{1}{2}}B_{q}^{\frac{1}{2}}$.

\smallskip

$2^\circ$) $\langle  \mathsf{f} \cdot (\mathsf{f} \cdot \nabla w), |w|^{q-2}(-b\cdot w) \rangle 
\leq \bar{I}_{q}^{\frac{1}{2}}B_{q}^{\frac{1}{2}}$.

\smallskip

$3^\circ$) $\langle \mu  u , |w|^{q-2} (-b \cdot w) \rangle  \leq \frac{\mu}{\mu-\omega_q} B_{q}^\frac{1}{2} \|w\|_q^\frac{q-2}{2}  \|f\|_q$ $\bigl(\text{here } \frac{2}{2-\sqrt{\delta}} < q \Rightarrow \|u\|_q \leq (\mu - \omega_q )^{-1} \|f\|_q \bigr)$.

\smallskip

$4^\circ$) $\langle b \cdot w, |w|^{q-2} b \cdot w \rangle = B_{q} .$

\smallskip

$5^\circ$) $\langle f, |w|^{q-2} (- b \cdot w)\rangle |\leq B_{q}^\frac{1}{2} \|w\|_q^\frac{q-2}{2} \|f\|_q .$

\smallskip

In $3^\circ$) and $5^\circ$) we estimate $B_{q}^\frac{1}{2} \|w\|_q^\frac{q-2}{2} \|f\|_q \leq \varepsilon_0 B_q + \frac{1}{4\varepsilon_0}\|w\|_q^{q-2} \|f\|^2_q$ ($\varepsilon_0>0$).

\smallskip

Therefore,
\begin{align*}
 \langle-b\cdot w, \phi \rangle \leq P_q^{\frac{1}{2}}B_{q}^{\frac{1}{2}} + |c|\bar{I}_{q}^{\frac{1}{2}}B_{q}^{\frac{1}{2}}+B_{q} + (q-2)B_{q}^{\frac{1}{2}}J_{q}^{\frac{1}{2}} + \varepsilon_0 B_q +  C_1(\varepsilon_0)\|w\|_q^{q-2} \|f\|^2_q.
\end{align*}
It is easily seen that $b \in \mathbf{F}_\delta$ 
is equivalent to the inequality
$$
\langle b^2 |\varphi|^2 \rangle \leq \delta \langle |\nabla \varphi|^2\rangle  + \lambda\delta\langle |\varphi|^2 \rangle, \quad \varphi \in W^{1,2}.
$$
Thus,
\begin{equation*}
B_q \leq \|b |w|^\frac{q}{2} \|_2^2 \leq \delta \| \nabla |w|^\frac{q}{2} \|_2^2 + \lambda\delta \|w\|_q^q = \frac{ q^2\delta}{4} J_q + \lambda\delta \|w\|_q^q.
\end{equation*}
Similarly, using that  $\nabla a \in \mathbf{F}_{\delta_a}$, we obtain
\begin{equation*}
P_q \leq \|(\nabla a) |w|^\frac{q}{2} \|_2^2 \leq \delta_a \| \nabla |w|^\frac{q}{2} \|_2^2 + \lambda\delta_a \|w\|_q^q = \frac{ q^2\delta_a}{4} J_q + \lambda\delta_a \|w\|_q^q.  
\end{equation*}
Then selecting $\varepsilon_0>0$ sufficiently small, and noticing that the assumption on $\delta$, $\delta_a$ in the theorem are strict inequalities, we can and will ignore below the terms multiplied by $\varepsilon_0$. The proof of Lemma \ref{b_est_lem} is completed.
\end{proof}

In \eqref{be2}, we apply Lemma \ref{b_est_lem} where the inequality $\frac{q\sqrt{\delta}}{2} J_q^{\frac{1}{2}}\bar{I}_q^{\frac{1}{2}} \leq \gamma_2\frac{q^2\delta}{4}J_q+\frac{1}{4\gamma_2}\bar{I}_q$, $\gamma_2>0$, is used. Thus, we have
\begin{align}
&\mu \|w\|_q^q + I_q^{} + c\bar{I}_{q}^{} + (q-2)(J_q + c\bar{J}_{q}^{}) \notag \\ 
&\leq |c| \biggl[\frac{q\sqrt{\delta_{\mathsf{f}}}}{4} J_q + \frac{q\sqrt{\delta_{\mathsf{f}}}}{4} I_q\biggr] + |c|(q-2) \frac{q\sqrt{\delta_{\mathsf{f}}}}{2} J_q  \label{be3} \\
&+ |c| \biggl[(\gamma\delta_{\mathsf{f}}+\gamma_2 \delta )\frac{q^2}{4}J_q + \biggl(\frac{1}{4\gamma}+\frac{1}{4\gamma_2}\biggr)\bar{I}_q\biggr] + |c|(q-2) \biggl[\gamma_1\delta_{\mathsf{f}} \frac{q^2}{4}J_q + \frac{1}{4\gamma_1}\bar{J}_q\biggr]  \notag \\
&+  \biggl[\big(\sqrt{\delta}\sqrt{\delta_a}+\delta\big)\frac{q^2}{4}+(q-2)\frac{q\sqrt{\delta}}{2}\biggr]J_q 
  + \mu_0\|w\|^q_q + C_1\|w\|_q^{q-2} \|f\|^2_q + \langle f,\phi\rangle, \notag
\end{align}
where $\mu_0:=|c|\lambda\sqrt{\delta_{\mathsf{f}}}\big(q^{-1} + \gamma \sqrt{\delta_{\mathsf{f}}})+|c|(q-2)\lambda\sqrt{\delta_{\mathsf{f}}}\big(q^{-1} + \gamma_1 \sqrt{\delta_{\mathsf{f}}} \big) + C_0$.

\medskip

\textbf{3.~}Let us prove that there exists constant $\eta>0$ such that
\begin{equation}
\label{star_ineq_}
\tag{$\ast$}
(\mu-\mu_0) \|w\|_q^q+\eta J_q \leq C_1\|w\|_q^{q-2}\|f\|_q^2+\langle f,\phi\rangle.
\end{equation}

Set $Q:=\big(\sqrt{\delta}\sqrt{\delta_a}+\delta\big)\frac{q^2}{4} + (q-2)\frac{q\sqrt{\delta}}{2}$.

\textbf{Case $c>0$}. 
First, suppose that $1-\frac{q\sqrt{\delta_{\mathsf{f}}}}{4}-\frac{q\sqrt{\delta}}{4} \geq 0$. We select $\gamma$, $\gamma_2>0$ such that $\frac{1}{4\gamma}+\frac{1}{4\gamma_2}=1$ while $\gamma\delta_{\mathsf{f}}+\gamma_2 \delta$ attains its minimal value. It is easily seen that $\gamma=\frac{1}{4}\bigl(1+\sqrt{\frac{\delta}{\delta_{\mathsf{f}}}} \bigr)$, $\gamma_2=\frac{1}{4}\bigl(1+\sqrt{\frac{\delta_{\mathsf{f}}}{\delta}} \bigr)$. We have $1-\frac{q\sqrt{\delta_{\mathsf{f}}}}{4} \geq 0$, and select $\gamma_1=\frac{1}{4}$.
Thus, the terms $\bar{I}_q$, $\bar{J}_q$  are no longer present in \eqref{be3}:
\begin{align*}
&\mu \|w\|_q^q + \biggl(1-c\frac{q\sqrt{\delta_{\mathsf{f}}}}{4} \biggr)I_q \\
&+  \biggl[q-2- c\frac{q\sqrt{\delta_{\mathsf{f}}}}{4} - c(q-2)\frac{q\sqrt{\delta_{\mathsf{f}}}}{2} - c(\delta_{\mathsf{f}} + 2 \sqrt{\delta_{\mathsf{f}}\delta} + \delta) \frac{q^2}{16}- c(q-2)\frac{q^2\delta_{\mathsf{f}}}{16} - Q\biggr]J_q \\
&\leq  \mu_0\|w\|_q^{q}+ C_1\|w\|_q^{q-2}\|f\|_q^2 + \langle f,\phi\rangle.
\end{align*}
By the assumptions of the theorem, $1-c\frac{q\sqrt{\delta_{\mathsf{f}}}}{4} \geq 0$, so by $J_q \leq I_q$ we obtain 
\begin{align*}
&\mu \|w\|_q^q + \biggl[q-1- c(q-1)\frac{q\sqrt{\delta_{\mathsf{f}}}}{2} - c(\delta_{\mathsf{f}} + 2 \sqrt{\delta_{\mathsf{f}}\delta} + \delta) \frac{q^2}{16}- c(q-2)\frac{q^2\delta_{\mathsf{f}}}{16} - Q\biggr]J_q \\
& \leq  \mu_0\|w\|_q^{q}+ C_1\|w\|_q^{q-2}\|f\|_q^2 + \langle f,\phi\rangle.
\end{align*}

Next, suppose that $1-\frac{q\sqrt{\delta_{\mathsf{f}}}}{4}-\frac{q\sqrt{\delta}}{4}<0$, but $1-\frac{q\sqrt{\delta_{\mathsf{f}}}}{4} \geq 0$. We select $\gamma=\frac{1}{q\sqrt{\delta_{\mathsf{f}}}}$, $\gamma_2=\frac{1}{q\sqrt{\delta}}$, and $\gamma_1=\frac{1}{4}$. Then the term $\bar{J}_q$ is no longer present, so using $\bar{I}_q \leq I_q$ we obtain
\begin{align*}
\mu \|w\|_q^q & + \biggl[1+c\biggl(1-\frac{q\sqrt{\delta_{\mathsf{f}}}}{2} - \frac{q\sqrt{\delta}}{4}  \biggr)\biggr]I_q \\
&+ \biggl[q-2-c\frac{q\sqrt{\delta_{\mathsf{f}}}}{4} - c(q-2)\frac{q\sqrt{\delta_{\mathsf{f}}}}{2} - c \frac{q\sqrt{\delta_{\mathsf{f}}} + q\sqrt{\delta}}{4} -c(q-2)\frac{q^2\delta_{\mathsf{f}}}{16} - Q \biggr]J_q \\
& \leq  \mu_0\|w\|_q^{q}+ C_1\|w\|_q^{q-2}\|f\|_q^2 + \langle f,\phi\rangle. \notag
\end{align*}
Thus, since $1+c\bigl(1-\frac{q\sqrt{\delta_{\mathsf{f}}}}{2} - \frac{q\sqrt{\delta}}{4}  \bigr) \geq 0$ by the assumptions of the theorem, we have using $J_q \leq I_q$
\begin{align*}
\mu \langle |w|^q \rangle & + \biggl[q-1+c - c \frac{q\sqrt{\delta}}{2}- c\frac{q^2\sqrt{\delta_{\mathsf{f}}}}{2}  -c(q-2)\frac{q^2\delta_{\mathsf{f}}}{16} - Q\biggr]J_q \\
& \leq  \mu_0\|w\|_q^{q}+ C_1\|w\|_q^{q-2}\|f\|_q^2 + \langle f,\phi\rangle, \notag
\end{align*}

Finally, suppose that $1-\frac{q\sqrt{\delta_{\mathsf{f}}}}{4} < 0$. We select $\gamma=\gamma_1=\frac{1}{q\sqrt{\delta_{\mathsf{f}}}}$, $\gamma_2=\frac{1}{q\sqrt{\delta}}$. Then using $\bar{I}_q \leq I_q$, $\bar{J}_q \leq J_q$ we obtain
\begin{align*}
\mu \|w\|_q^q & + \biggl[1+c\biggl(1-\frac{q\sqrt{\delta_{\mathsf{f}}}}{2} - \frac{q\sqrt{\delta}}{4}  \biggr)\biggr]I_q + \biggl[q-2+c(q-2)\bigg(1-\frac{q\sqrt{\delta_{\mathsf{f}}}}{4}\bigg) \\
&-c\frac{q\sqrt{\delta_{\mathsf{f}}}}{4} - c(q-2)\frac{q\sqrt{\delta_{\mathsf{f}}}}{2} - c \frac{q\sqrt{\delta_{\mathsf{f}}} + q\sqrt{\delta}}{4} -c(q-2)\frac{q\sqrt{\delta_{\mathsf{f}}}}{4} - Q \biggr]J_q \\
&  \leq  \mu_0\|w\|_q^{q}+ C_1\|w\|_q^{q-2}\|f\|_q^2 + \langle f,\phi\rangle. \notag
\end{align*}
Since $1+c\bigl(1-\frac{q\sqrt{\delta_{\mathsf{f}}}}{2} - \frac{q\sqrt{\delta}}{4}  \bigr) \geq 0$ by the assumptions of the theorem, we have using $J_q \leq I_q$
\begin{align*}
\mu \|w\|_q^q & + \biggl[q-1+c(q-1) - c\frac{q\sqrt{\delta}}{2} - c(q-1)q\sqrt{\delta_{\mathsf{f}}} - Q \biggr]J_q \\
& \leq  \mu_0\|w\|_q^{q}+ C_1\|w\|_q^{q-2}\|f\|_q^2 + \langle f,\phi\rangle, \notag
\end{align*}
In all three cases, the coefficient of $J_q$ is positive.
We have proved \eqref{star_ineq_}.

\smallskip

\textbf{Case $c<0$}. In \eqref{be3}, select $\gamma=\gamma_1=\frac{1}{q\sqrt{\delta_{\mathsf{f}}}}$, $\gamma_2=\frac{1}{q\sqrt{\delta}}$:
\begin{align*}
&\mu \|w\|_q^q + \bigg(1-|c|\frac{q\sqrt{\delta_{\mathsf{f}}}}{4}\bigg)I_q^{}  \\
& + \biggl[q-2- |c|(q-1)\frac{q\sqrt{\delta_{\mathsf{f}}}}{2}-|c|(q-2)\frac{q\sqrt{\delta_{\mathsf{f}}}}{4}   - |c|\frac{q\sqrt{\delta}}{4} - Q\biggr]J_q \notag   \\
&- |c|\bigg(1+\frac{q\sqrt{\delta_{\mathsf{f}}}}{4}+\frac{q\sqrt{\delta}}{4}\bigg)\bar{I}_{q}^{} - |c|(q-2)\bigg(1+\frac{q\sqrt{\delta_{\mathsf{f}}}}{4}\bigg)\bar{J}_{q}^{}\leq  \mu_0\|w\|_q^{q}+ C_1\|w\|_q^{q-2}\|f\|_q^2 + \langle f,\phi\rangle.
\end{align*} 
Using $I_q \geq \bar{I}_q$, $J_q \geq \bar{J}_q$, we obtain
\begin{align*}
&\mu \|w\|_q^q + \bigg(1- |c|\bigg(1+\frac{q\sqrt{\delta_{\mathsf{f}}}}{2}+\frac{q\sqrt{\delta}}{4}\bigg) \bigg)I_q^{}  \\
& + \biggl[q-2- |c|(q-1)\frac{q\sqrt{\delta_{\mathsf{f}}}}{2}-|c|(q-2)\frac{q\sqrt{\delta_{\mathsf{f}}}}{4}- |c|\frac{q\sqrt{\delta}}{4} -|c|(q-2)\bigg(1+\frac{q\sqrt{\delta_{\mathsf{f}}}}{4}\bigg)  - Q \biggr]J_q \notag \\
& \leq  \mu_0\|w\|_q^{q}+ C_1\|w\|_q^{q-2}\|f\|_q^2 + \langle f,\phi\rangle.
\end{align*}
By the assumptions of the theorem, $1- |c|\big(1+\frac{q\sqrt{\delta_{\mathsf{f}}}}{2}+\frac{q\sqrt{\delta}}{4}\big) \geq 0$. Therefore, by $I_q \geq J_q$,
\begin{align*}
&\mu \|w\|_q^q  + \biggl[q-1  - |c|\bigg(1+\frac{q\sqrt{\delta_{\mathsf{f}}}}{2}+\frac{q\sqrt{\delta}}{4}\bigg) \\
&- |c|(q-1)\frac{q\sqrt{\delta_{\mathsf{f}}}}{2}-|c|(q-2)\frac{q\sqrt{\delta_{\mathsf{f}}}}{4}- |c|\frac{q\sqrt{\delta}}{4} -|c|(q-2)\bigg(1+\frac{q\sqrt{\delta_{\mathsf{f}}}}{4}\bigg)  - Q \biggr]J_q \notag \\
& \leq  \mu_0\|w\|_q^{q}+ C_1\|w\|_q^{q-2}\|f\|_q^2 + \langle f,\phi\rangle,
\end{align*}
where the coefficient of $J_q$ is strictly positive by the assumptions of the theorem. We have proved \eqref{star_ineq_}.

\smallskip

\textbf{4.~}We estimate the term $\langle f, \phi \rangle$ by Lemma \ref{f_est_lem}: For each $\varepsilon_0>0$ there exists a constant $C=C(\varepsilon_0)<\infty$  such that 
$$
\langle f, \phi \rangle \leq \varepsilon_0 I_q + C \|w\|_q^{q-2} \|f\|^2_q.
$$

We choose $\varepsilon_0>0$  so small that in the estimates below we can ignore 
$\varepsilon_0 I_q$.  

Then \eqref{star_ineq_} yields the inequalities
$$ \|\nabla u_n\|_{q} \leq  K_1(\mu-\mu_0)^{-\frac{1}{2}} \|f\|_q, \quad K_1:=(C+C_1)^{\frac{1}{2}}, $$
$$\|\nabla u_n \|_{q j} \leq  K_2(\mu-\mu_0)^{\frac{1}{q}-\frac{1}{2}} \|f\|_q, \quad K_2:=C_{S}\eta^{-\frac{1}{q}}(q^2/4)^{\frac{1}{q}} (C+C_1)^{\frac{1}{2}-\frac{1}{q}},$$
where $C_S$ is the constant in the Sobolev Embedding Theorem.

If $c>0$ then $\delta_1=\delta<1$. If $c<0$ then 
elementary arguments show that, by the assumptions of the theorem, $\delta_1=(1-|c|)^{-2} \delta<1$. 
Therefore, \cite[Theorem 3.5]{KiS}
($
(\mu+\Lambda_q(a,b))^{-1}=s{\mbox-}L^q{\mbox-}\lim_n (\mu+\Lambda_q(a_n,b_n))^{-1}
$)
yields \eqref{reg_double_star}.
The proof of Theorem \ref{thm:reg} is completed.

\section{The iteration procedure}
\label{iteration_sect}

The following is a direct extension of the iteration procedure in \cite{KS}. Let $a \in (H_{u})$.

Recall that a vector field $b:\mathbb R^d \rightarrow \mathbb R^d$ belongs to $\mathbf F_{\delta_1}(A)$, $\delta_{1}>0$, the class of form-bounded vector fields (with respect to $A \equiv A_2:= [-\nabla \cdot a \cdot \nabla\upharpoonright C_c^\infty]^{\rm clos}_{2 \to 2}$), if $b_a^2:=b \cdot a^{-1} \cdot b
 \in L^1_\loc$ and there exists a constant $\lambda=\lambda_{\delta_1}>0$ such that 
$\|b_a(\lambda + A)^{-\frac{1}{2}}\|_{2\rightarrow 2}\leq \sqrt{\delta_1}.$

Consider $$\{a_n\}_{n=1}^\infty \subset  [C^1]^{d \times d} \cap (H_{u,\sigma,\xi})$$
 and 
$$\{b_n\}_{n=1}^\infty \subset [C^1]^d \cap \bigcap_{m \geq 1}\mathbf{F}_{\delta_1}(A^m), \quad \delta_1<4, \quad \lambda \neq \lambda(n,m).$$ Here $A^m \equiv A(a_m)$.

By \cite[Theorem 3.2]{KiS}, $-\Lambda_r(a_n,b_n):=\nabla \cdot a_n \cdot \nabla - b_n\cdot \nabla$, 
$D(\Lambda_r(a_n,b_n))=W^{2,r}$, is the generator of a positivity preserving $L^\infty$ contraction quasi contraction $C_0$ semigroup on $L^r$, $r \in \big]\frac{2}{2-\sqrt{\delta_1}},\infty[$, with the resolvent set of $-\Lambda_r(a_n,b_n)$ containing $\mu>\omega_r:=\frac{\lambda\delta_1}{2(r-1)}$ for all $n \geq 1$.

Set
$u_n := (\mu + \Lambda_r(a_n, b_n))^{-1} f$, $f \in L^1 \cap L^\infty$ and $g:=u_m- u_n$.

\begin{lemma}
\label{lem1}
There are positive constants $C=C(d), k=k(\delta_1)$ such that
\[
\|g \|_{r j} \leq \big( C \sigma^{-1} (\delta_1  + 2\xi\sigma^{-1})(1+2\xi) \|\nabla u_m \|_{q j}^2 \big)^\frac{1}{r} \big(r^{2k} \big)^\frac{1}{r} \|g\|_{x^\prime (r-2)}^{1-\frac{2}{r}},
\]
where $q \in \big]\frac{2}{2-\sqrt{\delta_1}} \vee (d-2), \frac{2}{\sqrt{\delta_1}}\big [, \; 2 x = q j, \; j=\frac{d}{d-2}, \; x^\prime := \frac{x}{x-1}$ and $x^\prime(r-2) > \frac{2}{2-\sqrt{\delta_1}}$, $\mu > \lambda_{\delta_1}$. 
\end{lemma}

The proof follows closely \cite[proof of Lemma 3.12]{KiS} or \cite[proof of Lemma 6]{KS}.

Iterating the inequality of Lemma \ref{lem1}, we arrive at

\begin{lemma}
\label{lem2}
In the notation of Lemma \ref{lem1}, assume that $\sup_m \|\nabla u_m \|_{q j}^2<\infty$, $\mu>\mu_0$. Then for any $r_0 > \frac{2}{2-\sqrt{\delta_1}}$ 
$$
\|g\|_\infty \leq B \|g \|_{r_0}^\gamma, \quad \mu \geq 1 + \mu_0 \vee \lambda_{\delta_1},
$$
where $\gamma = \big( 1 - \frac{x^\prime}{j} \big) \big( 1 - \frac{x^\prime}{j} + \frac{2 x^\prime}{r_0}\big)^{-1} > 0$, and $B=B(d, \delta_1)<\infty$.
\end{lemma}

The proof repeats \cite[proof of Lemma 3.13]{KiS} or \cite[proof of Lemma 7]{KS}.

\begin{remark*}
The assumption $\sup_m \|\nabla u_m \|_{q j}^2<\infty$ in Lemma \ref{lem2} is crucial and holds e.g.\,in the assumptions of Theorem \ref{thm:reg}(\textit{ii}). 
\end{remark*}

\section{Proof of Theorem \ref{thm:feller}}

By Lemma \ref{lem2}
and the second inequality in \eqref{reg_double_star}, we have for all
$r_0 > \frac{2}{2-\sqrt{\delta_1}}$
$$
\|u_n-u_m\|_\infty \leq B \|u_n-u_m \|_{r_0}^\gamma, \quad \mu \geq 1 + \mu_0 \vee \lambda_{\delta_1},
$$
where $\gamma>0$, $B<\infty$, and $u_n := (\mu + \Lambda_{r_0}(a_n, b_n))^{-1} f$, $f \in L^1 \cap L^\infty$.
By \cite[Theorem 3.5]{KiS},
$$
(\mu+\Lambda_{r_0}(a,b))^{-1}=s{\mbox-}L^{r_0}{\mbox-}\lim_n (\mu+\Lambda_{r_0}(a_n,b_n))^{-1},
$$
so  $\{u_n\}$ is fundamental in $C_\infty$.

\begin{lemma}
\label{lem5}
$s \mbox{-} C_\infty \mbox{-} \lim_{\mu \uparrow \infty} \mu (\mu + \Lambda_{C_\infty}(a_n,b_n))^{-1} = 1$ uniformly in $ n.$
\end{lemma}

The proof follows closely \cite[proof of Lemma 3.16]{KiS}.

We are in position to complete the proof of Theorem \ref{thm:feller}.
The assertion (\textit{i}) follows from the fact that $\{u_n\}$ is fundamental in $C_\infty$ and Lemma \ref{lem5} 
by applying the Trotter Approximation Theorem. 
(\textit{ii}) is Theorem \ref{thm:reg}\eqref{reg_double_star}. The proof of (\textit{iii}) is standard.
The proof of Theorem \ref{thm:feller} is completed.

\begin{remark*}
The arguments of the present paper extend more or less directly to the time-dependent case $\partial_t - \nabla \cdot a(t,x) \cdot \nabla + b(t,x)\cdot \nabla$, cf.\,\cite{Ki}. 
\end{remark*}

\end{document}